\documentclass[10pt]{amsart}
\usepackage{amssymb}
\usepackage{graphicx}
\usepackage{amsmath}
\usepackage{amsthm}
\usepackage{indentfirst}
\usepackage{mathrsfs}
\usepackage{hyperref}
\usepackage{xcolor}
\usepackage{csquotes}
\usepackage{url,hyperref}
\usepackage{enumerate}
\usepackage{mathtools}
\usepackage[normalem]{ulem}
\usepackage{cancel}

\hypersetup{
	unicode=false,          % non-Latin characters in Acrobat's bookmarks
	pdftoolbar=true,        % show Acrobat's toolbar?
	pdfmenubar=true,        % show Acrobat's menu?
	pdffitwindow=false,     % window fit to page when opened
	pdfstartview={FitH},    % fits the width of the page to the window
	pdftitle={My title},    % title
	pdfauthor={Author},     % author
	pdfsubject={Subject},   % subject of the document
	pdfcreator={Creator},   % creator of the document
	pdfproducer={Producer}, % producer of the document
	pdfkeywords={keyword1} {key2} {key3}, % list of keywords
	pdfnewwindow=true,      % links in new window
	colorlinks=true,       % false: boxed links; true: colored links
	linkcolor=blue,          % color of internal links (change box color with linkbordercolor)
	citecolor=blue,        % color of links to bibliography
	filecolor=magenta,      % color of file links
	urlcolor=blue           % color of external links
}

\newtheorem{theorem}{Theorem}[section]
\newtheorem{lemma}[theorem]{Lemma}

\theoremstyle{definition}
\newtheorem{definition}[theorem]{Definition}

\newtheorem{coro}[theorem]{Corollary}

\theoremstyle{remark}
\newtheorem{remark}[theorem]{Remark}
\numberwithin{equation}{section}

\newcommand{\thmref}[1]{Theorem~\ref{#1}}

\newcommand{\lemref}[1]{Lemma~\ref{#1}}
\newcommand{\secref}[1]{Sect. \ref{#1}}

\newcommand{\etc}{\textit{etc}.}

\newcommand{\ie}{\textit{i.e.}}
\newcommand{\eg}{\textit{e.g.}}

\newcommand{\ud}{\,\mathrm{d}}

\newcommand{\RR}{\mathbb{R}}

\newcommand{\Real}{\mathbb{R}}

\makeatletter
\newcommand*{\rom}[1]{\expandafter\@slowromancap\romannumeral #1@}
\makeatother
\newcommand{\vect}[1]{\boldsymbol{#1}}

\newcommand{\wt}[1]{\widetilde{#1}}

\newcommand{\eps}{\epsilon}

\DeclarePairedDelimiterX{\infdiv}[2]{(}{)}{%
  #1\;\delimsize\|\;#2%
}

\newcommand{\Abs}[1]{\left\lvert#1\right\rvert}

\newcommand{\Inner}[2]{\left\langle#1, #2\right\rangle}

\def\bigl{\mathopen\big}

\def\bigr{\mathclose\big}

\newcommand{\eqrefn}[1]{\eqref{#1}}

\newcommand{\id}{\mathbf{I}_d}
\newcommand{\statdist}{p_\infty} %
\newcommand{\fp}{Fokker-Planck}
\newcommand{\ren}{R{\'e}nyi}
\newcommand{\renyi}{\ren{}}
\newcommand{\poin}{Poincar{\'e}}

\newcommand*\Laplace{\Delta}

\newcommand{\divg}{\text{div}}
\newcommand{\mani}{\mathcal{M}}
\newcommand{\tang}{\mathscr{T}}
\newcommand{\muap}{\mu_{\alpha, p}}
\newcommand{\muaptime}[1]{\mu_{\alpha,p_{#1}}}

\newcommand{\gradt}{\text{grad}}
\newcommand{\InnerAlpha}[2]{\left\langle#1, #2\right\rangle_{\alpha, p}}

\newcommand{\rela}[2]{D\infdiv{#1}{#2}}
\newcommand{\relafisheralpha}[3]{\mathscr{I}_{#3}\infdiv{#1}{#2}}
\newcommand{\relafisher}[2]{\relafisheralpha{#1}{#2}{}}
\newcommand{\renyidivg}[2]{D_{\alpha}\infdiv{#1}{#2}}
\newcommand{\renyidivgalpha}[3]{D_{#3}\infdiv{#1}{#2}}
\newcommand{\metrictensor}[3]{g_{\alpha, #3}\left(#1, #2\right)}

\newcommand{\gw}{Gr{\"o}nwall's inequality}
\newcommand{\lsi}{log-Sobolev inequality}

\newcommand{\lb}{Lindblad}
\newcommand{\waittime}{\tau}
\newcommand{\decayfactor}{C}

\newcommand{\ou}{Ornstein-Uhlenbeck}

\newcommand{\bbren}{Benamou-Brenier}

\newcommand\myeq[1]{\mathrel{\stackrel{\makebox[0pt]{\mbox{\normalfont\tiny #1}}}{=}}}

\newcommand\myle[1]{\mathrel{\stackrel{\makebox[0pt]{\mbox{\normalfont\tiny #1}}}{\le}}}

\begin{document}

\title[Exponential decay of R{\'e}nyi divergence]{Exponential decay of
  R{\'e}nyi divergence under Fokker-Planck equations}

 \author[Y. Cao]{Yu Cao}
 \address[Y. Cao]{Department of Mathematics, Duke University, Box
 	90320, Durham NC 27708, USA}
\email{yucao@math.duke.edu}

 \author[J. Lu]{Jianfeng Lu}
 \address[J. Lu]{Department of Mathematics, Department
 	of Physics, and Department of Chemistry, Duke University, Box
 	90320, Durham, NC 27708, USA}
\email{jianfeng@math.duke.edu}

 \author[Y. Lu]{Yulong Lu}
 \address[Y. Lu]{Department of Mathematics, Duke University, Box
 	90320, Durham NC 27708, USA}
\email{yulonglu@math.duke.edu}

% \subjclass is required.
%%%%%%%%%%%%%%%%%%%%%%%%%%%%%%%%%%%%%%%%%%%%%%%%%%%%%%%%%%%%
% TODO: NEED TO UNCOMMENT.
%%%%%%%%%%%%%%%%%%%%%%%%%%%%%%%%%%%%%%%%%%%%%%%%%%%%%%%%%%%%

\date{\today}

\begin{abstract}
  We prove the exponential convergence to the equilibrium, quantified
  by \renyi{} divergence, of the solution of the Fokker-Planck
  equation with drift given by the gradient of a strictly convex
  potential. This extends the classical 
   exponential decay result on the relative entropy for the same
  equation.
\end{abstract}

\maketitle
%\tableofcontents

\section{Introduction and main results}
We consider  the long time behavior of the following Fokker-Planck equation on $\RR^d$
\begin{equation}\label{eq:fp}
\begin{split}
& \partial_t p_t(x) = \divg \left(p_t(x) \nabla V(x)  \right) + \Laplace p_t(x),\\
\end{split}
\end{equation}
where $V(x)$ is a smooth potential function on $\RR^d$, and
the initial datum $p_0$ is smooth and decays sufficiently fast at infinity. 
It is well-known \cite{markowich_trend_1999,toscani1999entropy}
that if $V$ satisfies the uniform convexity
(or Bakry-\'Emery \cite{bakry1985diffusions}) condition:
\begin{equation}\label{eq:BE}
D^2 V(x) \geq K \cdot \id \text{ for every } x\in \RR^d
\end{equation}
with some constant $K > 0$, then the solution $p_t$ of the Fokker-Planck
equation dissipates the relative entropy (or Kullback-Leibler
divergence, or free energy functional) exponentially fast towards the Gibbs stationary
distribution
$$
p_\infty(x) =e^{-V(x)},
$$ 
where we assume that the normalization constant is one without loss of generality. More precisely, 
\begin{equation}\label{eq:entropydecay}
\rela{p_t}{p_\infty} \leq  \rela{p_0}{p_\infty} e^{-2K t}, 
\end{equation}
where we recall that the relative entropy $\rela{p}{q}$ is defined by 
\begin{equation*}
\rela{p}{q} := 
\begin{cases}
  \displaystyle \int \frac{p}{q} \ln \left(\frac{p}{q}\right)\ud q,  & p \ll q, \\
  \infty, & \text{otherwise.}
\end{cases}
\end{equation*}
For convenience, we will abuse notation and use symbols $p$, $q$,
\etc, to represent probability measures as well as the density
functions associated with them. Whether a symbol refers to a
probability measure or a density should be clear from the context. In
addition, to avoid technicalities, all probability density functions
under consideration will be assumed to be smooth.

The exponential decay \eqrefn{eq:entropydecay} can be established by
the entropy dissipation method, which usually relies on the validity of
the \lsi{} with respect to $\statdist$. In fact, the entropy production
(time-derivative of entropy) is
\begin{equation}\label{eq:entropyprod}
\frac{\ud}{\ud t} \rela{p_t}{p_\infty} = - \relafisher{p_t}{\statdist},
\end{equation}
where $\relafisher{\cdot}{\cdot}$ is the \emph{relative Fisher information} defined by
\begin{equation}
\label{eqn::relative_Fisher}
\relafisher{p}{q} := \int \Bigl\lvert \nabla \ln \Bigl(\frac{p}{q}\Bigr) \Bigr\rvert^2 \ud p,
\end{equation}
if $p \ll q$ and $q\ll p$; 
otherwise, set $\relafisheralpha{p}{q}{} := \infty$.
We say that the measure $p_\infty = e^{-V}$ 
satisfies  
the \lsi{} (LSI) \cite{bakry2014,gross_hypercontractivity_1975,gross_logarithmic_1975}
with constant $K > 0$  if for all probability measures $p \ll \statdist$, we have
\begin{equation}\label{eq:lsi}
\rela{p}{\statdist} \leq \frac{1}{2K} \relafisher{p}{\statdist}.
\end{equation}
Then  \eqrefn{eq:entropydecay} follows directly  from
\eqrefn{eq:entropyprod}, \eqrefn{eq:lsi} and \gw{}.
By linearizing the LSI near $\statdist$, it yields the \poin{} inequality \cite{Otto00_log_sobolev}: if $\int f\ud \statdist = 0$, then
\begin{equation}
\label{eqn::poin_ineq}
K \int f^2\ud \statdist \le \int \Abs{\nabla f}^2\ud \statdist.
\end{equation}
The entropy dissipation method exemplified as above has become an important tool to  study convergence to equilibrium of 
solutions of evolutionary partial differential equations. 
We refer interested readers
to \cite{villani2003topics,villani2008entropy, villani2008optimal} for more extensive discussion on this method.

There is a huge amount of literature, attempting to generalize the above picture, by considering various semigroup dynamics and entropy measures. 
In \cite{arnold_convex_2001}, Arnold \emph{et al.}~considered more general \fp{} equations and \emph{admissible relative entropies} generated by convex functions $\psi$; more explicitly, the admissible relative entropy considered there has the form 
\begin{equation}
\label{eqn::admissible_relative_entropy}
D_{\psi} \infdiv{p}{q} := \int \psi\left(\frac{p}{q}\right) \ud q.
\end{equation}
It recovers the standard relative entropy by choosing $\psi(x) = x\log(x)$.
It is worthwhile to mention that an important special instance of this family of divergence is the \emph{Tsallis divergence} \cite{tsallis_possible_1988,nielsen_closed-form_2012} with order $\alpha\in [1,2]$, which refers to the choice $\psi_{\alpha}(x) = \frac{x^{\alpha} - 1}{\alpha - 1}$. 
It has been proved in \cite[Theorem 2.16]{arnold_convex_2001} that under certain assumptions, the solution of the \fp{} equation converges to its stationary distribution exponentially fast, quantified by the admissible relative entropy.

The decay of the solution of Fokker-Planck equation in relative
entropy can be viewed from a different, yet deeper perspective.  This
dates back to the celebrated work by Jordan, Kinderlehrer and Otto
\cite{JKO}, in which the Fokker-Planck equation is regarded as the
gradient flow of the relative entropy with respect to the
2-Wasserstein distance in the space of probability measures.  Based on
identifying a Riemannian structure on the Wasserstein space of
probability measures, Otto \cite{Otto01_porous} showed that a large
number of evolution equations could also be viewed as the gradient
flow in the 2-Wasserstein metric for certain energy functionals.
Moreover, from this geometric perspective, the strong geodesic
convexity of the functionals gives rise to a number of functional
inequalities, including the LSI; see \eg, \cite{Otto00_log_sobolev}.
By now, similar results in this direction have been obtained in
various scenarios, such as finite Markov chains \cite{MAAS20112250},
discrete porous medium equations \cite{erbar2014gradient},
quantum Fermionic Fokker-Planck equation \cite{carlen_analog_2014},
and quantum Lindblad equation \cite{CARLEN20171810}, just to name a
few.

\subsection*{Motivation and main results}
In this paper, 
we study the dissipation behavior of the solution of
\fp{} equation
with respect to the \renyi{} divergence \cite{renyi_measures_1961,Erven}, including
the relative entropy as a special instance.
\renyi{} divergence has been widely used in, for instance, coding \cite{harremoes_interpretations_2006}, statistics \cite{shayevitz_renyi_2011,begin_pac-bayesian_2016}, rare events \cite{atar_robust_2015,dupuis_sensitivity_2018}.
The precise definition of \renyi{} divergence
is given as follows.
\begin{definition}[\renyi{} divergence]
	For two probability distributions $p\ll q$, \renyi{} divergence is defined as
	\begin{equation}
	\label{eqn::renyi}
	\renyidivg{p}{q} = \left\{
	\begin{split}
      & \frac{1}{\alpha-1} \ln \left( \int \left(\frac{p}{q}\right)^{\alpha} \ud q\right), &&\qquad 0 < \alpha < \infty,\, \alpha \neq 1;\\
      &\int \frac{p}{q} \ln \left(\frac{p}{q}\right)\ud q, &&\qquad \alpha = 1.
	\end{split}\right.
	\end{equation}
	If $p$ is not absolutely continuous with respect to $q$, simply set $\renyidivg{p}{q} = \infty$.
\end{definition}
With fixed smooth distributions $p, q$ such that $p\ll q$, \renyi{} divergence is continuous with respect to order $\alpha$, in particular, $\renyidivgalpha{p}{q}{1} = \lim_{\alpha\rightarrow 1} \renyidivgalpha{p}{q}{\alpha}$. Both \renyi{} divergence and Tsallis divergence generalize the relative entropy, though in different flavors \cite{masi_step_2005}. 

It is important to notice that \renyi{} divergence \eqref{eqn::renyi} does not fit into the framework of admissible relative entropy \eqref{eqn::admissible_relative_entropy} studied in \cite{arnold_convex_2001}. To the best of our knowledge, \renyi{} divergence has not been used as an entropy measure to study the behavior of \fp{} equations. 
However, there is one related work \cite{carrillo2014renyi} utilizing the \enquote{relative \renyi{} entropy} (slightly different from \renyi{} divergence defined above) to obtain refined long time asymptotics of the solution of the porous medium equation to its Barenblatt profile.

Our interest in \renyi{} divergence is mainly motivated by the recent work on the second
laws of quantum thermodynamics \cite{Brandao}, which states that a
family of free energies -- quantum \renyi{} divergences (including sandwiched \renyi{} divergences \cite{muller-lennert_quantum_2013,wilde_strong_2014} and Petz-\renyi{} divergences \cite{petz_quasi-entropies_1986}), never increases
during state transition at microscopic level. 
One important challenge at the
quantum level is that there is no consensus about how quantum \renyi{} divergence should
be defined due to the noncommutative nature of quantum systems: two 
incompatible definitions of quantum \renyi{} divergence can be found in \cite{muller-lennert_quantum_2013,wilde_strong_2014, petz_quasi-entropies_1986} and the recent progress on more general $(\alpha, z)$-\renyi{} divergence can be found in \cite{audenaert_alpha-z-renyi_2015,carlen_inequalities_2018}.
For sandwiched \renyi{} divergence, 
Frank and Lieb rigorously proved that
for orders $\alpha\ge \frac{1}{2}$,
it is monotonically decreasing under all
completely positive trace preserving (CPTP) maps (\ie, data processing inequality holds) \cite{frank2013monotonicity}.  In particular, this
implies that the sandwiched \renyi{} divergence decreases under \lb{}
equation, which is generally viewed as the quantum analog of
Fokker-Planck equation.  There are some attempts to characterize the
convergence rate for sandwiched \renyi{} divergence under \lb{} equations
\cite{MullerHermes2018sandwichedrenyi}.
Motivated by the increasing attention to quantum \renyi{} divergence, 
we pull ourselves back from (quantum) microscopic dynamics to (classical) macroscopic time-evolution, and
examine the decay rate of the (classical) \renyi{} divergence, defined in
\eqrefn{eqn::renyi}, under the time-evolution of the Fokker-Planck equation \eqref{eq:fp}. The companion paper \cite{nextpaper} examines the convergence of the solution of  primitive Lindblad equations with GNS-detailed balance, quantified by the sandwiched \renyi{} divergence, whose analysis is strongly informed by the present paper.

 The main result of the present paper is the following theorem. 
 \begin{theorem}
	\label{thm::renyi_conv}
	Assume that $V$ satisfies \eqref{eq:BE}.
	Fix $\alpha \in (0, \infty)$ and a smooth initial probability distribution $p_0$ which decays sufficiently fast at infinity. 
	Let $p_t$ be the solution of the \fp{} equation \eqrefn{eq:fp}. Then
	there exists $\waittime \ge 0$ and $\decayfactor > 0$ such that
	\begin{equation}\label{eq:renyidecay}
	\renyidivg{p_t}{ \statdist} \leq \decayfactor \renyidivg{p_0}{ \statdist}e^{-2K t},\qquad \text{ for any } t \ge \waittime,
	\end{equation}
	 where $\waittime$ is given by
	\begin{equation}\label{eq:t}
	\waittime = \left\{
	\begin{split}
	& 0,  & \qquad \alpha\in (0,2]; \\
	& \frac{1}{2K}\ln(\alpha-1) ,  &\qquad  \alpha\in (2,\infty); \\
	\end{split}\right.
	\end{equation}
	and $\decayfactor$ is given by
	\begin{equation*}
	\decayfactor = \left\{
	\begin{split}
	& \frac{\renyidivgalpha{p_0}{\statdist}{1}}{\renyidivg{p_0}{\statdist}}, &\alpha\in(0,1]; \\
	& \frac{e^{\renyidivgalpha{p_0}{\statdist}{2}} - 1}{\renyidivg{p_0}{\statdist}}, &\alpha \in (1,2]; \\
	& (\alpha-1)\frac{e^{\renyidivgalpha{p_0}{\statdist}{2}} - 1 }{\renyidivg{p_0}{\statdist}}, &\alpha\in (2,\infty). \\ 
	\end{split}\right.
	\end{equation*}

\end{theorem}

\begin{remark}
 Observe that when $\alpha = 1$, \thmref{thm::renyi_conv} recovers 
 the classical dissipation estimate  \eqrefn{eq:entropydecay}
 for relative entropy under the \fp{} equation. The quantum analog of \thmref{thm::renyi_conv} for the Lindblad equation can be found in \cite[Theorem 1.8]{nextpaper}.
\end{remark}

The proof of Theorem \ref{thm::renyi_conv} is presented in Section
\ref{sec:proof}. Let us explain here briefly the key strategies in our
proof. First, we prove this theorem for the case $\alpha=2$, by using
the \poin{} inequality \eqref{eqn::poin_ineq}.  Next, for the range
$\alpha\in(0,2)$, the exponential decay follows immediately from the
monotonicity of the \renyi{} divergence with respect to order $\alpha$
(see \lemref{lem:monotone}).  Then, we present a comparison lemma (see
\lemref{lem::renyi_equiv}), which bounds
$\renyidivgalpha{p_{T}}{\statdist}{\alpha_1}$ above by
$\renyidivgalpha{p_0}{\statdist}{\alpha_0}$ for $\alpha_1 > \alpha_0$,
at the expense of marching time $T$.  Finally, this comparison lemma
is used to prove the exponential decay of the \renyi{} divergence with
order $\alpha\in (2,\infty)$.

Before the proof of Theorem \ref{thm::renyi_conv}, we also show in the
next section a new gradient flow structure of the Fokker-Planck
equation based on the \renyi{} divergence, which will facilitate the
proof of Theorem \ref{thm::renyi_conv}. To the best of our knowledge,
this gradient flow structure does not seem to fit into any existing
framework like \cite{Otto01_porous} and has interest in its own right.

\subsection*{Contribution}
We prove the exponential convergence of the solution of the \fp{} equation \eqref{eq:fp} to the Gibbs stationary distribution $\statdist$, quantified by the \renyi{} divergence \eqref{eqn::renyi} in \thmref{thm::renyi_conv}. 
The proof of \thmref{thm::renyi_conv} for \fp{} equation has very similar quantum analog for Lindblad equation (see \cite{nextpaper}), which suggests a possibility of having almost parallel approaches to study \fp{} equation and \lb{} equation. In addition, we show that under certain metric tensor \eqref{eqn::metric_tensor}, the \fp{} equation can be, at least formally, identified as the gradient flow dynamics of the \renyi{} divergence (see \secref{sec:gradflow}), which can be of independent interest.

The rest of this paper is organized as follows.  In
\secref{sec:gradflow}, we first show that the Fokker-Planck equation
can be formally viewed as the gradient flow of \renyi{} divergence under a certain metric
tensor.  \secref{sec:proof} is fully devoted into the proof of Theorem
\ref{thm::renyi_conv}.

\section{\fp{} equation as the gradient flow of \renyi{} divergence}
\label{sec:gradflow}

This section aims to identify the \fp{} equation \eqref{eq:fp} 
as the gradient flow of \renyi{} divergence for any order $\alpha \in (0,\infty)$, with respect to a certain metric tensor \eqrefn{eqn::metric_tensor} 
in the space of probability measures,
which generalizes the well-known fact that \fp{} equation
is the $L^2$-Wasserstein gradient flow of the relative entropy \cite{JKO}.
Interested readers may refer to 
\eg, \cite{ambrosio2008gradient, Otto00_log_sobolev, Otto01_porous} 
for extensive treatment of gradient flows in the space of probability measures. The derivations in this section are formal and follow closely with those  in \cite{Otto01_porous}.

We first define a Riemannian structure on a space
of probability measures under which the
gradient flow of $\renyidivg{\cdot }{\statdist}$ gives the Fokker-Planck equation \eqref{eq:fp}. 
By Riemannian structure, we mean a manifold (denoted by $\mani_{\alpha}$)
 and a metric tensor, denoted by $\metrictensor{\cdot}{\cdot}{p}$, defined on the tangent space $\tang_{p} \mani_{\alpha}$. The dependence
 of the metric tensor on $\alpha$ and $p$ will be clear in the sequel. For a fixed
 Riemannian structure $(\mani_{\alpha}, \metrictensor{\cdot}{\cdot}{p})$,
 the gradient of the energy functional 
$\renyidivg{\cdot }{\statdist}$ at $p\in \mani_\alpha$ is defined  as the element in $\tang_{p}\mani_{\alpha}$, 
denoted by $\gradt{} D_{\alpha}|_{p}$ or simply $\gradt{} D_{\alpha}$ (when no confuse arises for $p$), such that
\begin{equation}
\label{eqn::gradient}
\metrictensor{\gradt{} D_{\alpha}}{\nu}{p}
= \left.\frac{\ud}{\ud \eps} \renyidivg{p + \eps \nu }{ \statdist}
\right\rvert_{\eps = 0} ,\qquad \forall\, \nu \in \tang_{p} \mani_{\alpha}.
\end{equation}
The corresponding gradient flow dynamics (of the \renyi{} divergence) is given by
\begin{equation}
\label{eqn::gradient_flow}
\partial_t p_t = -\gradt D_{\alpha}|_{p_t}.
\end{equation}
Below we specify the space $\mani_{\alpha}$ and define the metric tensor $\metrictensor{\cdot}{\cdot}{p}$.

Let $\mani_{\alpha}$ be the space of smooth
probability distributions, which  have finite \renyi{} divergence with respect to $\statdist$, \ie,
\begin{equation*}
\mani_{\alpha} := \left\{\ p \ll \statdist\ \text{ is smooth } |\ \renyidivg{p}{\statdist} < \infty \right\}.
\end{equation*}
We will not delve into technical details of the differential structure of
the manifold and think of the tangent space $\tang_{p} \mani_{\alpha}$
at $p\in \mani_{\alpha}$ as
\begin{equation*}
  \tang_{p} \mani_{\alpha} = \bigl\{\text{signed functions } \nu \text{ on $\RR^d$ with } \int \nu(x)\ \ud x = 0\bigr\}.
\end{equation*}
For any $\nu \in \tang_{p} \mani_{\alpha}$, let $\Psi_\nu$ be a weak solution
to the equation
\begin{equation}
\label{eqn::nu_Psi}
\nu + \divg(p \nabla \Psi_{\nu}) = 0.
\end{equation}
Namely, for all smooth and compactly supported test functions $f$, we have
\begin{equation*}
\int f\nu\ \ud x  = \int \nabla f \cdot \nabla \Psi_{\nu}\ud p.
\end{equation*}
Note that  $\Psi_{\nu}$ is defined uniquely up to some additive constant.
Then whenever dealing with an element $\nu\in \tang_{p}\mani_{\alpha}$, it is equivalent to consider its associated $\Psi_{\nu}$.

In order to define the metric tensor $\metrictensor{\cdot}{\cdot}{p}$,
we also need to introduce an inner product $\InnerAlpha{\cdot}{\cdot}$
on the space of vector fields. More precisely,
we define, for vector fields $\vect{U} = \begin{pmatrix}
u_1 & u_2 & \cdots & u_d
\end{pmatrix}$ and $\vect{V} = \begin{pmatrix}
v_1 & v_2 & \cdots & v_d
\end{pmatrix}$ where $u_j$ and $v_j$ are functions on $\Real^d$, for all $1\le j \le d$, 
the inner product $\InnerAlpha{\cdot}{\cdot}$ by
\begin{equation}
\label{eqn::InnerAlpha}
\InnerAlpha{\vect{U}}{\vect{V}} := 
\sum_{j=1}^{d} \alpha \int u_j  v_j\ud \muap,
\end{equation}
where $\muap$ is a probability distribution defined by 
\begin{equation}
\muap := \frac{\displaystyle\Bigl(\frac{p}{\statdist}\Bigr)^{\alpha} \statdist }{ 
	\displaystyle \int \Bigl(\frac{p}{\statdist}\Bigr)^{\alpha}\ud \statdist}.
\end{equation}
With this inner product, we define the metric tensor $\metrictensor{\cdot}{\cdot}{p}$ by 
\begin{equation}
\label{eqn::metric_tensor}
\metrictensor{\nu_1}{\nu_2}{p}
:= \InnerAlpha{\nabla \Psi_{\nu_1}}{\nabla \Psi_{\nu_2}},
\end{equation}
for $\nu_k \in \tang_p \mani_{\alpha}$, where $\Psi_{\nu_k}$
are related to $\nu_k$ via \eqref{eqn::nu_Psi}, \ie,
\[\nu_k + \divg(p \nabla \Psi_{\nu_k}) = 0, \qquad k = 1, 2.\]

When $\alpha=1$, it is easy to see that $\muap = p$ and the resulting metric tensor reduces to the one in \cite{Otto01_porous}.

Finally, we check that the Fokker-Planck equation is indeed the  gradient flow of 
$\renyidivg{\cdot }{\statdist}$ with respect to the Riemannian structure defined above.
In fact, by the definition of the gradient \eqref{eqn::gradient}, we have from
direct computations that
\begin{align}
\label{eqn::compute_renyi_divg_deri}
\begin{split}
\metrictensor{\gradt{} D_{\alpha}}{\nu}{p} &= \left.\frac{\ud}{\ud \eps} \renyidivg{p + \eps \nu }{ \statdist} \right\rvert_{\eps = 0} \\
&= \frac{\alpha}{\alpha-1} \frac{ \int \left(\frac{p}{\statdist}\right)^{\alpha-1} \ud \nu}{\int \left(\frac{p}{\statdist}\right)^{\alpha} \ud \statdist } \\
&\myeq{\eqref{eqn::nu_Psi}}  -\frac{\alpha}{\alpha-1} \frac{ \int \left(\frac{p}{\statdist}\right)^{\alpha-1} \divg(p \nabla \Psi_{\nu})\ud x}{\int \left(\frac{p}{\statdist}\right)^{\alpha} \ud \statdist } \\
&= \alpha \frac{\int \left(\frac{p}{\statdist}\right)^{\alpha-1} \nabla \left(\frac{p}{\statdist} \right) \cdot \nabla \Psi_{\nu} \ud \statdist }{\int \left(\frac{p}{\statdist}\right)^{\alpha} \ud \statdist } \\
&= \InnerAlpha{-\nabla \phi}{\nabla \Psi_{\nu}},
\end{split}
\end{align}
where 
\begin{equation}
\label{eqn::sphi}
\phi := -\ln(p/\statdist)= -\ln(p) - V.
\end{equation}
In view of the definition of metric tensor \eqref{eqn::metric_tensor}, we have
\begin{align*}
\gradt{} D_{\alpha} + \divg\left(p \bigl(-\nabla \phi\bigr)\right) = 0. 
\end{align*}
Consequently, the corresponding gradient flow dynamics is
\begin{equation}
\partial_t p_t = -\gradt D_{\alpha}|_{p_t} = -\divg(p_t \nabla \phi_t),
\end{equation}
where $\phi_t = -\ln(p_t/\statdist)$. This  exactly recovers the Fokker-Planck equation in \eqrefn{eq:fp}.

An immediate consequence from gradient flow structure is the monotonicity of the \renyi{} divergence under the \fp{} dynamics, which is summarized in the following corollary.
\begin{coro}
\renyi{} divergence $\renyidivg{p_t}{\statdist}$ is monotonically decreasing with respect to time $t$ if $p_t$ solves the \fp{} equation \eqref{eq:fp}.
\end{coro}

\begin{proof}
Since $p_t$ solves the \fp{} equation, then $\Psi_{\nu}$ in \eqref{eqn::compute_renyi_divg_deri}, in fact, equals $\phi_t \equiv -\ln\left(p_t/\statdist\right)$. Thus, \eqref{eqn::compute_renyi_divg_deri} becomes
\begin{equation}
\label{eqn::deri_renyidivg}
\frac{\ud}{\ud t} \renyidivg{p_t}{\statdist} = -\Inner{\nabla \phi_t}{\nabla \phi_t}_{\alpha, p_t} \le 0.
\end{equation}
\end{proof}

With such a Riemannian structure, it is natural to define $(\alpha,r)$-\emph{Wasserstein distance} via \emph{\bbren{}} formalism \cite{benamou_computational_2000}: 
\begin{equation}
W_{\alpha,r} (p, q) := \inf_{\gamma_{\cdot}: [0,1]\rightarrow \mani_\alpha, \gamma_0 = p, \gamma_1 = q} \left(\int_{0}^{1} \sqrt{\metrictensor{\dot{\gamma}_s}{\dot{\gamma}_s}{\gamma_s}}^r \ud s\right)^{1/r}.
\end{equation}
Interested readers may refer to \cite{villani2008optimal,ambrosio2008gradient,dolbeault_new_2009} for rigorous treatment of Wasserstein distance.
We will not pursue the properties of this distance measure herein and leave it for future research.

\section{Proof of Theorem \ref{thm::renyi_conv}}\label{sec:proof}
We will first introduce a concept called \emph{relative $\alpha$-Fisher information}.
The proof of Theorem \ref{thm::renyi_conv} is divided into 
three following-up subsections according to three regimes of $\alpha$.

\subsection{Relative $\alpha$-Fisher information}
Suppose $p_t$ solves the \fp{} equation \eqref{eq:fp} and let us define
\begin{equation}
\label{eqn::Phi}
\begin{split}
\Phi_t &:= -\ln\left(\frac{\muaptime{t}}{\statdist}\right) \\
&= -\ln \left(\left(\frac{p_t}{\statdist}\right)^{\alpha}\right) + (\alpha-1) \renyidivg{ p_t }{ \statdist} \\
&= \alpha \phi_t + (\alpha-1) \renyidivg{ p_t }{ \statdist }.
\end{split}
\end{equation}
Then the time derivative of \renyi{} divergence is linked to the relative Fisher information, using \eqref{eqn::deri_renyidivg}, \eqref{eqn::InnerAlpha} and \eqref{eqn::Phi},
\begin{align}
\label{eqn::deri_renyi}
\begin{split}
\frac{\ud }{\ud t} \renyidivg{ p_t }{ \statdist } &= - \frac{1}{\alpha} \int  \Abs{\nabla \Phi_t}^2\ud \muaptime{t} \\
&= -\frac{1}{\alpha} \relafisher{\muaptime{t}}{ \statdist} \\
&=: - \relafisheralpha{p_t}{\statdist}{\alpha},\\
\end{split}
\end{align}
where we introduce \emph{relative $\alpha$-Fisher information} $\relafisheralpha{p}{\statdist}{\alpha} := \frac{1}{\alpha} \relafisher{\muap}{\statdist}$. Note that the relative $\alpha$-Fisher information $\relafisheralpha{p}{\statdist}{\alpha}$ generalizes the relative Fisher information  $\relafisher{p}{\statdist}$ in \eqref{eqn::relative_Fisher}.

\subsection{Case (\rom{1}): $\alpha = 2$}
\label{subsec::alpha=2}
First, we will show that relative $2$-Fisher information can be bounded below by the \renyi{} divergence with order $2$. The quantum analog of the following lemma is provided in \cite[Prop. 4.3]{nextpaper}.
\begin{lemma}[Uniform lower bound of relative $2$-Fisher information]
	Suppose that $\renyidivgalpha{p}{\statdist}{2} < \infty$, then 
	\begin{equation}
	\label{eqn::2-fisher-bound}
	\relafisheralpha{p}{\statdist}{2}  \ge 2K \left(1-e^{-\renyidivgalpha{p}{\statdist}{2}}\right).
	\end{equation}
\end{lemma}
\begin{proof}
	Let $\eps = \sqrt{\int \frac{(p - \statdist)^2}{\statdist}\ud x}$ and let $f = (p - \statdist)/\eps$. Thus we know that $p = \statdist + \eps f$, $\int \frac{ f^2}{\statdist}\ud x = 1$ and $\int f\ud x = 0$. Then, with some straightforward calculation,
\begin{align*}
\renyidivgalpha{p}{\statdist}{2} &= \ln(1+\eps^2), \\
\relafisheralpha{p}{\statdist}{2} &= \frac{2\eps^2}{1+\eps^2} \int \Abs{\nabla \left(\frac{f}{\statdist}\right)}^2\ud \statdist.
\end{align*}
By \poin{} inequality \eqref{eqn::poin_ineq},
\begin{equation*}
\relafisheralpha{p}{\statdist}{2} \ge \frac{2\eps^2}{1+\eps^2} K \int \left(\frac{f}{\statdist}\right)^2\ud \statdist = \frac{2K\eps^2}{1+\eps^2} = 2K \left(1-e^{-\renyidivgalpha{p}{\statdist}{2}}\right).
\end{equation*}
\end{proof}

\begin{remark}
	In the above proof, $\eps^2 \equiv \int \frac{(p - \statdist)^2}{\statdist}\ud x$ turns out to be the well-known $\chi^2$-divergence $\chi^2(p,\statdist)$. It is straightforward to observe that $\renyidivgalpha{p}{\statdist}{2} = \ln\left(1+\chi^2(p,\statdist)\right)$.
\end{remark}

\begin{proof}[Proof of \thmref{thm::renyi_conv} in {\bf Case (\rom{1})}]
By \eqref{eqn::deri_renyi} and \eqref{eqn::2-fisher-bound}, we immediately have
\begin{align*}
\frac{\ud }{\ud t} \renyidivgalpha{ p_t }{ \statdist }{2} &= -\relafisheralpha{p_t}{\statdist}{2} \\
&\le -2K \left(1-e^{-\renyidivgalpha{p_t}{\statdist}{2}}\right).
\end{align*}
Then,
\begin{align*}
\frac{\ud}{\ud t} \left(\ln \left(e^{\renyidivgalpha{p_t}{\statdist}{2}} - 1\right) \right) \le -2 K.
\end{align*}
After integrating both sides from time $0$ to $t$ and after some straightforward simplification, we have 
\begin{equation}
\label{eqn::renyidivg_2_bound}
\begin{split}
\renyidivgalpha{p_t}{\statdist}{2} &\le \ln\left(1 + (e^{\renyidivgalpha{p_0}{\statdist}{2}} - 1) e^{-2Kt}\right) \le (e^{\renyidivgalpha{p_0}{\statdist}{2}} - 1) e^{-2Kt}\\
&= C_2 \renyidivgalpha{p_0}{\statdist}{2} e^{-2K t}. \\
\end{split}
\end{equation}
where $C_2 = \frac{e^{\renyidivgalpha{p_0}{\statdist}{2}} - 1}{\renyidivgalpha{p_0}{\statdist}{2}}$. Apparently, $\waittime_2 = 0$.
\end{proof}

\subsection{Case  (\rom{2}): $\alpha \in (0,2)$}
We first recall a useful lemma on the monotonicity of \renyi{} divergence with respect to the order $\alpha$.
\begin{lemma}{\cite[Theorem 3]{Erven}}
\label{lem:monotone}
Let $p, q$ be two probability distributions. Then $D_\alpha\infdiv{p}{q}$ is non-decreasing with respect to the order $\alpha\in (0, \infty)$.
\end{lemma}

\begin{proof}[Proof of \thmref{thm::renyi_conv} in {\bf Case (\rom{2})}] 
 Thanks to Lemma \ref{lem:monotone} and \eqref{eqn::renyidivg_2_bound}, 
 we immediately have, for $\alpha \in (0,2]$ and for all $t\ge 0$, that
 \begin{equation*}
 \begin{aligned}
 D_\alpha \infdiv{p_t}{p_{\infty}}  &\leq  D_2 \infdiv{p_t}{p_{\infty}}
   \leq (e^{\renyidivgalpha{p_0}{\statdist}{2}} - 1) e^{-2Kt} \\
   &= \frac{e^{\renyidivgalpha{p_0}{\statdist}{2}} - 1}{\renyidivg{p_0}{\statdist}} \renyidivg{p_0}{\statdist} e^{-2Kt}\\
 \end{aligned}
\end{equation*}
thus
$\decayfactor_{\alpha} = \frac{e^{\renyidivgalpha{p_0}{\statdist}{2}} - 1}{\renyidivg{p_0}{\statdist}}$. Apparently, waiting period
$\waittime_{\alpha} = 0$.

Recall that we also have the exponential decay of the relative entropy \eqref{eq:entropydecay} due to the LSI. Then by similar argument for $\alpha\le 1$,
\begin{align*}
\renyidivg{p_t}{p_{\infty}}  &\leq  \renyidivgalpha{p_t}{p_{\infty}}{1} \le \renyidivgalpha{p_0}{\statdist}{1} e^{-2K t} \\
&= \frac{\renyidivgalpha{p_0}{\statdist}{1}}{\renyidivg{p_0}{\statdist}}\renyidivg{p_0}{\statdist} e^{-2Kt}.\\
\end{align*}
Thus $\wt{C}_{\alpha} = \frac{\renyidivgalpha{p_0}{\statdist}{1}}{\renyidivg{p_0}{\statdist}}$.
To compare $\wt{C}_\alpha$ and $C_{\alpha}$ when $\alpha\le 1$, notice that 
\begin{align*}
e^{\renyidivgalpha{p_0}{\statdist}{2}} - 1 \ge \renyidivgalpha{p_0}{\statdist}{2} \ge \renyidivgalpha{p_0}{\statdist}{1}.
\end{align*}
Therefore, $C_{\alpha} \ge \wt{C}_{\alpha}$, which suggests that using $\wt{C}_{\alpha}$ provides a better bound for the prefactor when $\alpha \le 1$. Summarizing the above results for the case $\alpha \in (0,1]$ and $\alpha\in (1,2]$ leads into the conclusion in \thmref{thm::renyi_conv} for $\alpha\in (0,2]$.
\end{proof}

\subsection{Case (\rom{3}): $\alpha \in (2, \infty)$} In this case, we would like to prove Theorem \ref{thm::renyi_conv} by utilizing the results for case (\rom{1}) (see \secref{subsec::alpha=2}), and a useful comparison lemma for the family of \renyi{} divergences $\{D_\alpha\infdiv{p_t}{p_\infty}\}_{\alpha > 1}$ when $p_t$ solves the \fp{} equation \eqrefn{eq:fp}.

\begin{lemma}[Comparison lemma]
	\label{lem::renyi_equiv}
	Let $1 < \alpha_0 < \alpha_1 < \infty$.
	If $p_t$ solves the Fokker-Planck equation \eqref{eq:fp} with initial condition $p_0$, then 
	\begin{equation}
      D_{\alpha_1} \infdiv{p_T}{\statdist} \le \frac{\alpha_1(\alpha_0-1)}{\alpha_0(\alpha_1-1)} D_{\alpha_0}\infdiv{p_0}{
        \statdist} \le \renyidivgalpha{p_0}{\statdist}{\alpha_0}, 
	\end{equation}
	where $ T = \frac{1}{2K} \ln\left(\frac{\alpha_1 - 1}{\alpha_0 - 1}\right)$.
	
\end{lemma}
 Lemma \ref{lem::renyi_equiv} states that the \renyi{} divergence $\renyidivg{p_t}{p_\infty}$ can be bounded from above by a \renyi{} divergence with a smaller order than $\alpha$ at the expense of marching time $T$. 
 A simpler version of Lemma \ref{lem::renyi_equiv} 
 for the \ou{} process ($V = \frac{|x|^2}{2}+\frac{d}{2}\ln(2\pi)$) was proved in \cite[Theorem 3.2.3]{CIT-064}. 
 Since we are unaware of the proof of this lemma in literature
 for the \fp{} equation with a strictly convex potential $V$, 
 we include a proof at the end of this section for completeness.

\begin{proof}[Proof of Theorem  \ref{thm::renyi_conv} in {\bf Case (III)}] 
For any $\alpha \in (2,\infty)$, let us consider time $t \ge T_2 := \frac{1}{2K}\ln(\alpha-1)$. By \lemref{lem::renyi_equiv} with $\alpha_0 = 2$ and $\alpha_1 = \alpha$, we have
\begin{align*}
\renyidivgalpha{p_t}{\statdist}{\alpha} &\le \renyidivgalpha{p_{t-T_2}}{\statdist}{2} \myle{\eqref{eqn::renyidivg_2_bound}} (e^{\renyidivgalpha{p_0}{\statdist}{2}} - 1) e^{-2K(t-T_2)}\\
&= \frac{(e^{\renyidivgalpha{p_0}{\statdist}{2}} - 1) e^{2KT_2}}{\renyidivg{p_0}{\statdist}} \renyidivg{p_0}{\statdist} e^{-2Kt}.
\end{align*}
Therefore, $C_{\alpha} = \frac{(e^{\renyidivgalpha{p_0}{\statdist}{2}} - 1) e^{2KT_2}}{\renyidivg{p_0}{\statdist}} = (\alpha-1)\frac{e^{\renyidivgalpha{p_0}{\statdist}{2}} - 1 }{\renyidivg{p_0}{\statdist}}$, and the waiting time $\waittime_{\alpha} = T_2 \equiv \frac{1}{2K}\ln(\alpha-1)$.
\end{proof}

\begin{proof}[Proof of Lemma \ref{lem::renyi_equiv}]
	First, we need a variant of the \lsi{} \eqrefn{eq:lsi}. Let $p$ in \eqref{eq:lsi} be $p = \frac{f \statdist}{\int f\ud \statdist}$ where $f$ is a  smooth, strictly positive function with $\int f\ud \statdist<\infty$. Then \eqrefn{eq:lsi} can be re-written as
	\begin{equation}
	\label{eqn::lsi_variant}
	\int f \ln(f)\ud \statdist - \left(\int f\ud \statdist\right) \ln \left(\int f\ud \statdist\right) \le \frac{1}{2K} \int \frac{\Abs{\nabla f}^2}{f}\ud \statdist.
	\end{equation}

	Then, we follow the proof of \cite[Theorem 3.2.3]{CIT-064}. 
	Let $\beta_t = 1 + (\alpha_0 -1) e^{2Kt}$ and define
	\begin{equation*}
	F_t = \ln\left( \int h_t^{\beta_t}\ud \statdist \right)^{\frac{1}{\beta_t}},
	\end{equation*}
	where $h_t := p_t/\statdist$. 
	It should be emphasized that both $\beta_t$ and $h_t$ are changing during the time evolution: the order $\beta_t$ is changing according to the above choice and the distribution $p_t$ is evolving following the Fokker-Planck equation. 
	We shall show that $F_t$ is non-increasing in time.
	In fact,
	\begin{equation*}
	\begin{split}
	\frac{\ud }{\ud t} F_t &=  \frac{1}{\beta_t^2} \left[ \beta_t \frac{ \frac{\ud}{\ud t} \int h_t^{\beta_t}\ud \statdist}{ \int h_t^{\beta_t} \ud \statdist} - \frac{\ud {\beta}_t}{\ud t} \ln \left(\int h_t^{\beta_t} \ud \statdist \right) \right]. \\
	\end{split}
	\end{equation*}
	To simplify the notation, denote $Z_t := \int h_t^{\beta_t}\ud \statdist$. Multiplying both sides of the last equation by $\beta_t^2 Z_t$ and rearranging a few terms
	\begin{equation*}
	\begin{split}
	\beta_t^2 Z_t \frac{\ud }{\ud t}F_t
	&= \frac{\ud {\beta}_t}{\ud t} \int h_t^{\beta_t} \ln (h_t^{\beta_t}) \ud \statdist + \beta_t^2 \int h_t^{\beta_t-1} \partial_t {h}_t \ud \statdist 
	- \frac{\ud {\beta}_t}{\ud t} Z_t  \ln Z_t \\
	&  \myle{\eqref{eqn::lsi_variant}} \frac{1}{2K} \frac{\ud {\beta}_t}{\ud t} \beta_t^2 \int h_t^{\beta_t-2} \Abs{\nabla h_t}^2 \ud \statdist  + \beta_t^2 \int h_t^{\beta_t-1}\left( -\divg(p_t\nabla \phi_t) \right) \ud x\\
	&= \frac{1}{2K} \frac{\ud {\beta}_t}{\ud t} \beta_t^2 \int h_t^{\beta_t-2} \Abs{\nabla h_t}^2 \ud \statdist 
	- \beta_t^2 (\beta_t-1) \int h_t^{\beta_t-2} \Abs{\nabla h_t}^2 \ud \statdist \\
	&= \beta_t^2 \int h_t^{\beta_t-2} \Abs{\nabla h_t}^2 \ud \statdist \left( \frac{1}{2K}\frac{\ud {\beta}_t}{\ud t} - (\beta_t-1)\right) = 0.
	\end{split}
	\end{equation*}
	
	Because $\beta_t > 0$ and $Z_t > 0$, $F_t$ is non-increasing.
	Therefore, $F_t \le F_0$, \ie, 
	\begin{equation}
	\renyidivgalpha{p_t}{\statdist}{\alpha_t} \le \frac{\beta_t}{\beta_t - 1} \frac{\beta_0 - 1}{\beta_0} 
	\renyidivgalpha{p_0}{\statdist}{\alpha_0}.
	\end{equation}
	Then the lemma is proved by choosing time $T$ such that $\beta_{T} = \alpha_1$, whence $T = \frac{1}{2K}\ln\left(\frac{\alpha_1 - 1}{\alpha_0 - 1}\right)$.
\end{proof}

\section*{Acknowledgment}
The work of YC and JL is supported in part by the National Science
Foundation under grant DMS-1454939.

\bibliographystyle{amsplain}
\bibliography{reference.bib}

% When $\alpha > 1$, the exponential decay of the \renyi{} divergence $D_{\alpha} \infdiv{p_t}{p_\infty}$ with a rate $2K$
%only kicks in after some waiting period $\waittime$; see Example~\ref{ex:quantumOU} below. The upper
%bound of the \wtt{} $\waittime$ given in \eqrefn{eq:t} may not be
%sharp.

%\begin{example}\label{ex:quantumOU}
%	Consider a one dimensional \ou{} process: dimension $d = 1$ and
%	$V(x) = \frac{x^2}{2} + \frac{1}{2}\ln(2\pi)$; then $K = 1$ and
%	$\statdist = N(0,1)$.  Choose the initial distribution $p_0$ to
%	be $N\left(1, \frac{1}{2}\right)$.  Then at time $t$,
%	$p_t= N\left( e^{-t}, 1 -\frac{1}{2} e^{-2t}\right)$.
%	A direct calculation leads to
%	\begin{align*}
%	\begin{aligned}
%	\renyidivg{p_t}{\statdist} = -\frac{1}{2}\ln\left(1 -
%	\frac{1}{2}e^{-2t}\right) - \frac{1}{2(\alpha-1)} & \ln\left(1 + \frac{1}{2}e^{-2t}(\alpha-1)\right) \\ 
%	& \qquad + \frac{\alpha}{(\alpha-1)+2 e^{2t}}.
%	\end{aligned}
%	\end{align*}
%	Hence as $t\rightarrow \infty$,
%	$\renyidivg{p_t}{\statdist} \approx \frac{\alpha e^{-2t}}{2}$, which
%	decays exponentially fast with rate $2K$ ($2K = 2$ in the current
%	case) for any $\alpha\in(0,\infty)$. However, when
%	$t \sim \mathcal{O}(1)$, this is not the case: in the third term
%	above, for large $\alpha \gg 1$, the additive term $\alpha-1$ might
%	dominate the term $2e^{2t}$ in the denominator; thus we cannot
%	observe exponential decay with rate close to $2K$. 
%	Therefore, in general
%	one needs a waiting period $t^*$ as in \thmref{thm::renyi_conv}
%	before achieving the asymptotic exponential decay. 
%\end{example}
%

\end{document}